\documentclass[a4paper,12pt]{article}

\usepackage{amssymb}
\usepackage{amsmath, amsthm, enumitem, mathrsfs}
\usepackage{graphicx}
\usepackage{caption}
\usepackage{blindtext}
\usepackage[paper=a4paper,dvips,top=2.5cm,left=2cm,right=2cm,
    foot=1.5cm,bottom=3cm]{geometry}

\newtheorem{theorem}{Theorem}
\newtheorem*{thm_nonum}{Theorem}

\newtheorem{claim}{Claim}
\newtheorem{proposition}{Proposition}

\newtheorem*{conj_nonum}{Conjecture}
\newtheorem{example}{Example}

\title{A Note on Integer Domination of Cartesian Product Graphs}

\author{K. Choudhary\\
\small Department of Mathematics and Statistics\\[-0.8ex]
\small Indian Institute of Technology Kanpur\\[-0.8ex]
\small Kanpur, India\\
\small \texttt{keerti.india@gmail.com}\\
\and
S. Margulies\\
\small Department of Mathematics\\[-0.8ex]
\small Pennsylvania State University\\[-0.8ex]
\small State College, PA\\
\small \texttt{margulies@math.psu.edu}\\
\and
I. V. Hicks\\
\small Department of Computational and Applied Mathematics\\[-0.8ex]
\small Rice University\\[-0.8ex]
\small Houston, TX\\
\small \texttt{ivhicks@rice.edu}
}

\begin{document}

\maketitle

\begin{abstract} 
Given a graph $G$, a dominating set $D$ is a set of vertices such that any vertex in $G$ has 
at least one neighbor (or possibly itself) in $D$. 
A $\{k\}$-dominating multiset $D_k$ is a multiset of vertices such that any vertex in $G$ has 
at least $k$ vertices from its closed neighborhood in $D_k$ when counted with multiplicity.
In this paper, we utilize the approach developed by Clark and Suen (2000) and  properties of 
binary matrices to prove a ``Vizing-like" inequality on minimum $\{k\}$-dominating multisets 
of graphs $G,H$ and the Cartesian product graph $G \Box H$. Specifically, denoting the size of 
a minimum $\{k\}$-dominating multiset as $\gamma_{\{k\}}(G)$, we demonstrate that 
$\gamma_{\{k\}}(G) \gamma_{\{k\}}(H) \leq  2k \ \gamma_{\{k\}}(G \Box H)$~.
\end{abstract}

\section{Introduction} Let $G$ be a simple undirected graph $G=(V,E)$ with vertex set $V$ and edge set $E$. The open neighborhood of a vertex $v\in V(G)$ is denoted by $N_{G}(v)$,
and the closed neighborhood of $v$ is denoted by $N_{G}[v]$.
A dominating set $D$ of a graph G is a subset of $V(G)$ such that for all $v \in V(G)$,
$N_{G}[v] \cap D \neq \emptyset$, and the size of a minimum dominating set is denoted by $\gamma(G)$. The Cartesian product of two graphs $G$ and $H$, denoted $G \Box H$, is the graph with
vertex set $V(G)\times V(H)$, where vertices $gh, g'h' \in V(G \Box H)$ are adjacent
whenever $g = g'$ and $(h,h')\in E(H)$, or $h = h'$ and $(g,g') \in E(G)$ (see Example \ref{ex_prodg_not}). 

In 1963, and again more formally in 1968, V. Vizing proposed a simple and elegant conjecture that has subsequently become one of the most famous open questions in domination theory.
\begin{conj_nonum}[Vizing \cite{vizing}, 1968]
Given graphs $G$ and $H$, $\gamma(G)\gamma(H) \leq \gamma(G \Box H)$. 
\end{conj_nonum}
Although easy to state, a definitive proof of Vizing's conjecture (or a counter-example) has remained elusive. Over the past forty years (see \cite{viz_survey_2009} and references therein), Vizing's conjecture has been shown to hold on certain restricted classes of graphs, and furthermore, upper and lower bounds on the inequality have been gradually tightened. Additionally, as numerous direct attempts on the conjecture have failed, research approaches expanded to include explorations of similar inequalities for total, paired, and fractional domination \cite{book_dom}. However, the most significant breakthrough occurred in 2000, when Clark and Suen \cite{viz_clark_suen} demonstrated that $\gamma(G)\gamma(H) \leq 2 \gamma(G \Box H)$. This ``Vizing-like" inequality immediately suggested similar inequalities for total \cite{ho_total_dom} and paired \cite{hou_jiang_pr_dom} domination (2008 and 2010, respectively). In 2011, Choudhary, Margulies and Hicks \cite{susan_ds} improved the inequalities from \cite{ho_total_dom, hou_jiang_pr_dom} for total and paired domination by applying techniques similar to those of Clark and Suen, and also specific properties of binary matrices.  In this paper, we explore \emph{integer domination} (or $\{k\}$-domination), and again generate an improved inequality with this combined technique.  

A \emph{multiset} is a set in which elements are allowed to appear more than once, e.g. $\{1, 2, 2\}$. All graphs and multisets in this paper are finite. A $\{k\}$-dominating multiset $D_k$ of a graph $G$ is a multiset of vertices of $V(G)$ such that, for each $v \in V(G)$, the number of vertices of $N_G[v]$ contained in $D_k$ (counted with multiplicity) is at least $k$. A $\gamma_{\{k\}}$-set of $G$ is a minimum
$\{k\}$-dominating multiset, and the size of a minimum $\{k\}$-dominating multiset is denoted by $\gamma_{\{k\}}(G)$. Additionally, note that a $\{1\}$-dominating multiset is equivalent to the standard dominating set. 

The notion of a $\{k\}$-dominating \emph{multiset} is equivalent to the more familiar notion of a $\{k\}$-dominating \emph{function}. The study of $\{k\}$-dominating functions was first introduced by Domke, Hedetniemi, Laskar, and Fricke \cite{intdom1} (see also \cite{book_dom2}, pg. 90), and further explored by Bre\v{s}ar, Henning and Klav\v{z}ar in \cite{intdom_ieq}. In \cite{hou_lu}, the authors investigate integer domination in terms of graphs with specific packing numbers, and in \cite{intdom_ieq}, various applications of $\{k\}$-dominating functions are described, such as physical locations (stores, buildings, etc.) serviced by up to $k$ fire stations (as opposed to the single required fire station in the canonical application of dominating sets). Finally, in \cite{intdom_ieq}, the authors prove the following ``Vizing"-like inequality:
\begin{thm_nonum} [\cite{intdom_ieq}] Given graphs $G$ and $H$, $\gamma_{\{k\}}(G)\gamma_{\{k\}}(H) \leq k(k+1) \gamma_{\{k\}}(G \Box H)$~.
\end{thm_nonum}
\noindent Observe that for $k = 1$, the  Bre\v{s}ar-Henning-Klav\v{z}ar theorem is equivalent to the bound proven by Clark and Suen. In this paper, we improve this upper bound from $O(k^2)$ to $O(k)$, and prove the following theorem:
\begin{thm_nonum} Given graphs $G$ and $H$, $\gamma_{\{k\}}(G)\gamma_{\{k\}}(H) \leq 2k \ \gamma_{\{k\}}(G \Box H)$~.
\end{thm_nonum}
\noindent Again, observe that for $k=1$, this theorem is equivalent to the bound proven by Clark and Suen. In the next section, we develop the necessary background and present the proof. 

\section{Background, Notation and Proof of Theorem} \label{sec_back} In this section, we introduce the necessary background and notation used throughout the paper, and prove several propositions to streamline the proof of Theorem \ref{thm_kdom}.

For $gh \in V(G\Box H)$, the $G$-neighborhood $\big($denoted by $N_{\underline{\mathbf{G}} \Box H}(gh)\big)$ and the $H$-neighborhood $\big($denoted by $N_{G \Box \underline{\mathbf{H}}}(gh)\big)$ are defined as follows:
\begin{align*}
N_{\underline{\mathbf{G}} \Box H}(gh) &= \{g'h \in V(G\Box H) ~\rvert~ g' \in N_G(g)\}~,\\ 
N_{G \Box \underline{\mathbf{H}}}(gh) &= \{gh' \in V(G\Box H) ~\rvert~ h' \in N_H(h)\}~. 
\end{align*}
Thus, $N_{\underline{\mathbf{G}} \Box H}(gh)$ and $N_{G \Box \underline{\mathbf{H}}}(gh)$ are both subsets of $V(G \Box H)$. Additionally, the edge set $E(G \Box H)$ can be partitioned into two sets (\textbf{G}-edges and \textbf{H}-edges) where
\begin{align*}
\text{\textbf{G}-edges} &= \{(gh,g'h) \in E(G\Box H) ~\rvert~ h \in V(H) \text{ and } ( g,g')\in E(G)\}~, \\
\text{\textbf{H}-edges} &= \{(gh,gh')\in E(G\Box H) ~\rvert~ g \in V(G) \text{ and } (h,h')\in E(H)\}~.
\end{align*}
Given a dominating set $D$ of a graph $G$ and a vertex $v \in V(G)$, we say that the vertices in $N_G[v] \cap D$ are the \emph{v-dominators} in $D$.

A \emph{union of multisets} is denoted by $\uplus$, e.g. $\{1,2,2\} \uplus \{1,2,3\} = \{1,1,2,2,2,3\}$. The union of a multiset
with itself $t$ times is denoted by $\uplus^t$, e.g. $\uplus^2 \{1,2,2\} = \{1,2,2\} \uplus \{1,2,2\} = \{1,1,2,2,2,2\}$.
The \emph{cardinality of a multiset} is equal to the summation over the number of occurrences
of each of its elements, e.g. $\big|\{1,2,2\}\big| = 3$, and given a multiset $A$, we denote the number of occurrences of a 
particular element $a$ in $A$ as $|A|_a$, e.g. $\big|\{1,2,2\}\big|_2 = 2$, and $\big|\{1,2,2\}\big|_4 = 0$.  
A multiset $B$ is a \emph{sub-multiset} of multiset A if each element $b \in B$ is present
in $A$, and $|B|_b \leq |A|_b$, e.g. $\{1, 2, 2\} \subseteq  \{1, 2, 2, 2\} \varsubsetneq \{1,2\}$. 
Finally, let $A$ be a multiset and $B$ a set. Then, $|A|_B = \sum_{b \in B} |A|_b$. For example, $\{1,1,2,5,6,6\}_{\{1,4,6\}} = 4$.

Given graphs $G$ and $H$, let $A \subseteq \uplus^t~V(G\Box H)$, where $t$ is any positive integer. When defining a multiset, we must not only describe the elements contained in the multiset, but also define the number of times a specific element appears in the multiset. Thus, the $\varPhi$-projection and $\varPsi$-projection of $A$ on graphs $G$ and $H$ are multisets defined as
\begin{align*}
\varPhi_{G}(A) &= \Big\{g \in V(G)~:~\exists ~h \in V(H) \text{ with } gh \in A, \text{ where }\big|\varPhi_{G}(A)\big|_g =  \max \big\{|A|_{gh}~:~h\in V(H)\big\} \Big\}~,\\
\varPhi_{H}(A) &= \Big\{h \in V(H)~:~\exists~g \in V(G) \text{ with } gh \in A, \text{ where }\big|\varPhi_{H}(A)\big|_h =  \max \big\{|A|_{gh}~:~g\in V(G)\big\}\Big\}~,\\
\varPsi_{G}(A) &= \Big\{g \in V(G)~:~\exists~h \in V(H) \text{ with } gh \in A, \text{ where }\big|\varPsi_{G}(A)\big|_g = \textstyle{\sum}_{h\in V(H)} |A|_{gh}\Big\}~,\\
\varPsi_{H}(A) &= \Big\{h \in V(H)~:~\exists ~g \in V(G) \text{ with } gh \in A, \text{ where }\big|\varPsi_{H}(A)\big|_h = \textstyle{\sum}_{g\in V(G)} |A|_{gh}\Big\}~.
\end{align*}
Note that multisets $\varPhi_{G}(A)$ and $\varPsi_{G}(A)$ contain identical elements, but the number of occurrences of a given $g$ in $\varPhi_{G}(A)$ is defined by a \emph{max}, whereas the number of occurrences of the same $g$ in $\varPsi_{G}(A)$ is defined by a \emph{sum}. This \emph{max/sum} distinction in these multiset definitions will play a critical role of our proof of Theorem \ref{thm_kdom}. We now present an example of $\varPhi_{G}(A)$ and $\varPsi_{G}(A)$.

\begin{example} \label{ex_prodg_not}  Consider graphs $G, H$ and $G \Box H$:
\vspace{-25pt}
\begin{figure}[ht]
\hspace{15pt}\begin{minipage}{0.2\linewidth}
\begin{center}
\includegraphics[scale=.25,trim=0 200 700 0]{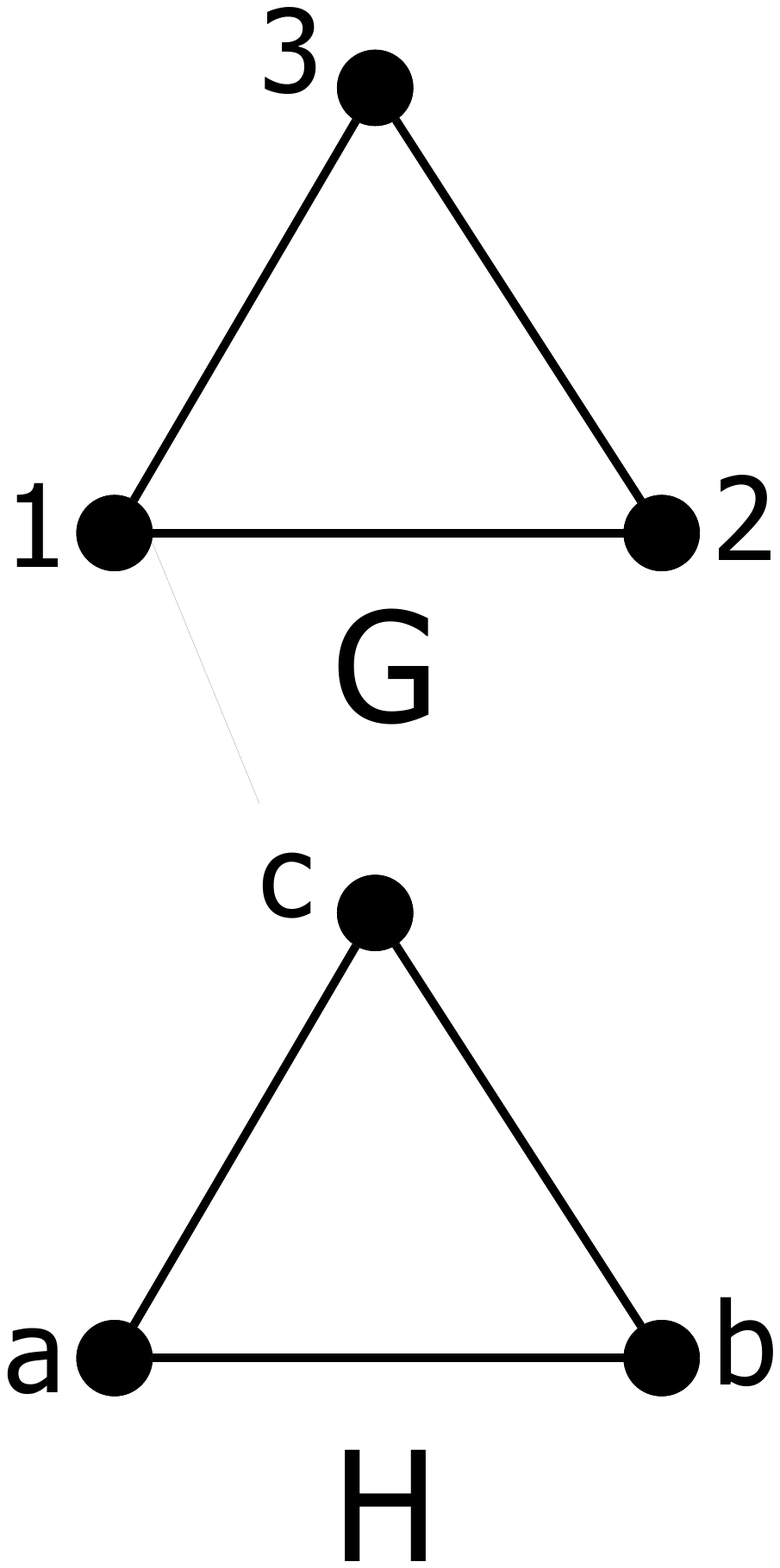}
\end{center}
\end{minipage}
\begin{minipage}{0.4\linewidth}
\begin{center}
\includegraphics[page=1,scale=.25,trim=0 540 490 225]{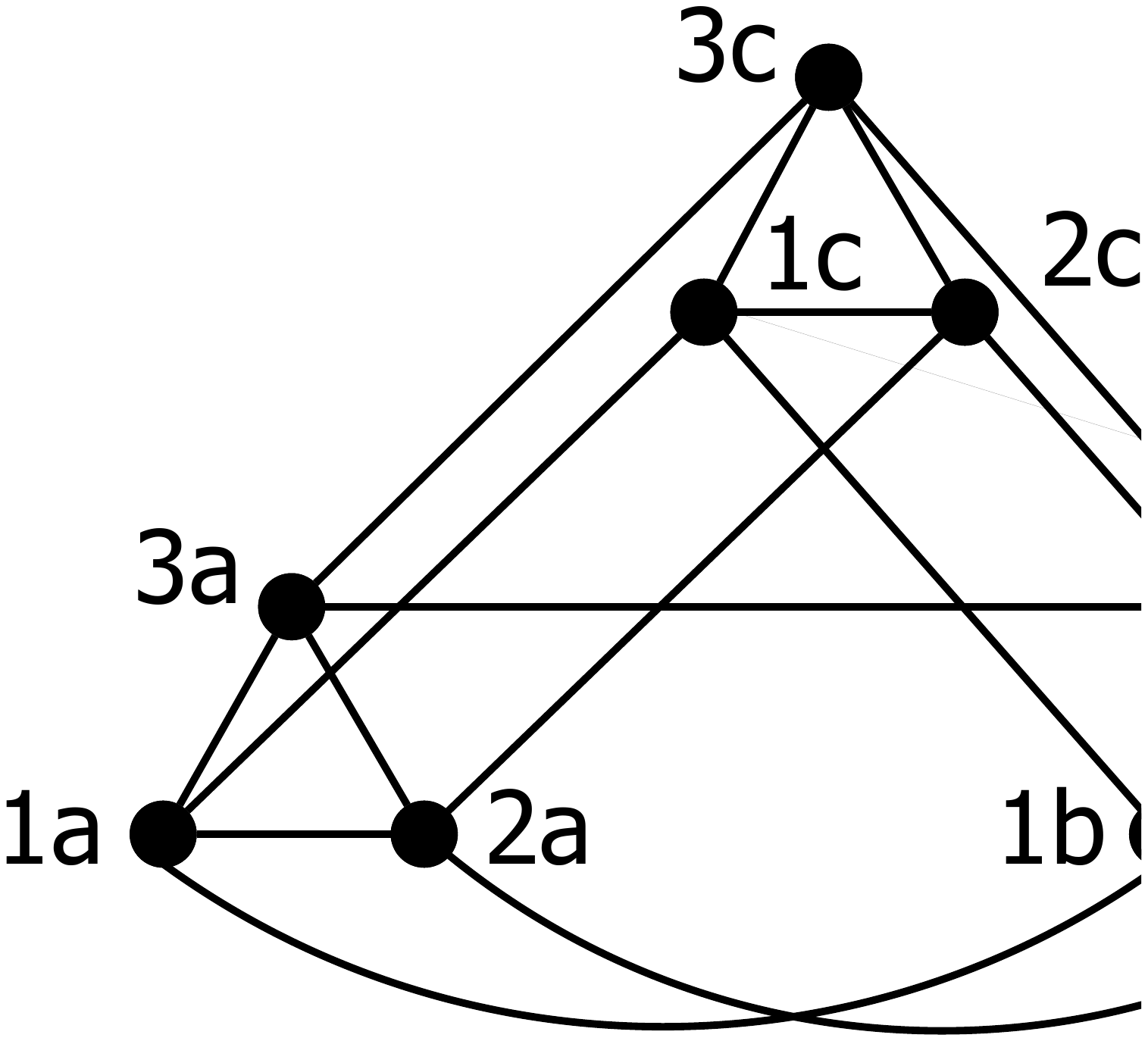}
\end{center}
\end{minipage}
\begin{minipage}{0.4\linewidth}
\begin{center}
\includegraphics[page=2,scale=.25,trim=402 540 200 225]{tri_tri_prod.pdf}
\end{center}
\end{minipage}
\caption*{(a) $G, H$ and $G \Box H$} \label{fig_GHprod}
\end{figure}

Let $A = \{1b, 1c, 1c, 2a, 2a, 2a, 2b\} \subseteq \uplus^3\hspace{1pt}V(G\Box H)$. Note $|A|_{1b} = 1, |A|_{1c} = 2, |A|_{2a} = 3$, and $|A|_{2b} = 1$. Then $\max\big\{|A|_{1h} : h \in V(H) \big\} = 2, \max\big\{|A|_{2h} : h \in V(H) \big\} = 3, \sum_{h \in V(H)}|A|_{1h} = 3$ and $\sum_{h \in V(H)}|A|_{2h} = 4$. Therefore, $\varPhi_{G}(A) = \{1,1,2,2,2\}$ and $\varPsi_{G}(A) = \{1,1,1,2,2,2,2\}$. \hfill $\Box$
\end{example}

Let $\{P_1, P_2, \ldots, P_t\}$ be a multiset of subsets of a set $A$. Then $P^A = \{P_1, P_2, \ldots, P_t\}$ is a \emph{$k$-partition} of $A$ if each element of $A$ is present in
exactly $k$ of the sets $P_1,\ldots, P_t$. For example, given $A=\{1,2,3,4,5,6,7\}$, then $P^A = \big\{\{1,2,3\}, \{1,2,3\}, \{4,5,6\}, \{6,7\},$ $\{5,7\}, \{4\}\big\}$ is a 2-partition of $A$. We observe that the subset $\{1,2,3\}$ is present twice in $P^A$, demonstrating that a $k$-partition can be a multiset.

\begin{example} \label{ex_kpart} Consider the following graph $G$:
\vspace{-10pt}
\begin{center}
\includegraphics[scale=.25,trim=0 500 0 75]{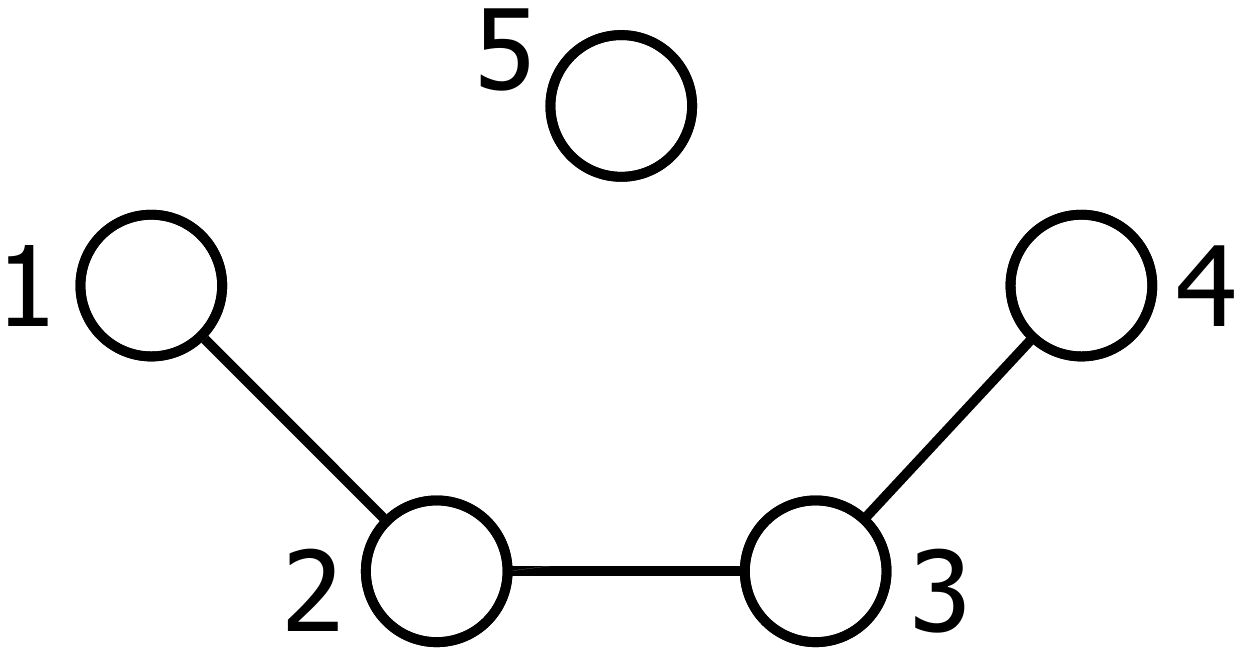}
\end{center}
\vspace{-10pt}
Here, a 2-partition of $V(G)$ is $P^G = \big\{\{1\}, \{1,2,3\}, \{2,3,4\}, \{4\}, \{5\}, \{5\}\big\} = \{P^G_1,\ldots,P^G_6\}$. Observe that each $v \in V(G)$ appears in exactly \emph{two} sets in $P^G$. Also, note that each $P^G_i$ is a \emph{set} (i.e., it contains no duplicated elements), but $P^G$ itself is a \emph{multiset}. Finally, observe that a minimum 2-dominating multiset of $G$ is  $\{1,2,3,4,5,5\}$, and thus  $\gamma_{\{2\}}(G) = 6$. \hfill $\Box$
\end{example}

Given a graph $G$, we will now define the concept of domination among multisets of $V(G)$. Given a positive integer $t$ and multisets $A, B \subseteq \uplus^t\hspace{1pt}V(G)$, we say that $A$ \emph{dominates} $B$ if,  for each $b\in B$, the number of vertices of $N_G[b]$
present in $A$ (counted with multiplicity) is \emph{at least} the number of occurrences of $b \in B$. In other words,
\begin{align*}
|A|_{N_G[b]} &\geq |B|_{b}~, \quad \quad \text{ for each $b\in B$}~.
\end{align*}
The following proposition can now be verified.

\begin{proposition} \label{prop_gmdom} Given graphs $G$, $H$, and positive integers $t,k$, with $t \geq k$, then 
$\phantom{xxx}$
\begin{enumerate}
	\item A multiset $A \subseteq \uplus^t\hspace{1pt}V(G)$ is a $\{k\}$-dominating multiset for $G$ if and only if $A$ dominates $\uplus^k\hspace{1pt}V(G)$.
	\item Given multisets $A, B, A', B' \subseteq \uplus^t\hspace{1pt}V(G)$, if $A$ dominates $A'$, and $B$ dominates $B'$,
then $A\uplus B$ dominates $A'\uplus B'$.
	\item Given multisets $A, A' \subseteq \uplus^t\hspace{1pt}V(G \Box H)$, if $A$ dominates $A'$, then $\varPsi_H (A)$ dominates $\varPsi_H (A')$.
\end{enumerate}
\end{proposition}

The proof of Prop. \ref{prop_gmdom} is a straight forward application of the definitions. Thus, we skip the proof for space considerations.

\begin{proposition} \label{prop_domrel} Given graphs $G, H$,  let $\{u_1 ,\ldots, u_{\gamma_{\{k\}}(G)}\}$ and  $\{\overline{u}_1 ,\ldots, \overline{u}_{\gamma_{\{k\}}(H)}\}$ be minimum $\{k\}$-dominating multisets of $G, H$, respectively, and let $P^G = \{P^G_1, \ldots , P^G_{\gamma_{\{k\}}(G)}\}$ and $P^H = \{P^H_1, \ldots ,$ $ P^H_{\gamma_{\{k\}}(H)}\}$be a $k$-partitions of $V(G), V(H)$ respectively, such that $u_i \in P^G_i$ and $P^G_i \subseteq N_G[u_i]$ (and $\overline{u}_j \in P^H_j$ and $P^H_{j}\subseteq N_H[\overline{u} _{j}]$, respectively) for $1 \leq i \leq \gamma_{\{k\}}(G)$ and   $1 \leq j \leq \gamma_{\{k\}}(H)$. 
\begin{enumerate}
	\item Let $I\subseteq \{1,\ldots,\gamma_{\{k\}}(G)\}$,
$A = \uplus_{i\in I} \{u_i\} $, and
$C = \uplus_{i\in I} P_i^G$. Then $A$ dominates $C$. Furthermore, given an integer ${k'} \geq k$, let $B \subseteq \uplus^{k'} \hspace{1pt}V(G)$
be any other multiset dominating $C$.
Then $\lvert B\lvert \geq \lvert A\lvert$.
	\item Let $J\subseteq \{1, \ldots, \gamma_{\{k\}}(H)\}$,
$\overline{A} = \displaystyle\uplus_{j\in J} \{ \overline{u}_j \} $, and
$\overline{C} = \uplus_{j\in J} P_j^H$. Then $\overline{A}$ dominates
$\overline{C}$. Furthermore,  given an integer ${k'} \geq k$, let $\overline{B} \subseteq \uplus^{k'} \hspace{1pt} V(H)$ be any other multiset dominating $\overline{C}$.
Then $\lvert \overline{B} \lvert \geq \lvert \overline{A} \lvert$.
\end{enumerate}
\end{proposition}

\noindent \emph{Proof of Prop. \ref{prop_domrel}.1:} We will first prove $A$ dominates $C$. Since $P^G_i \subseteq N_G[u_i]$, $u_i$ dominates $P^G_i$. Therefore, $\uplus_{i\in I} \{u_i\}$ dominates $\uplus_{i\in I} P_i^G$,
i.e. $A$ dominates $C$. Now let multiset $W = \uplus_{i\notin I} \{u_i\} $. Since $B$ is any multiset dominating $C$, $B \uplus W$ dominates $\Big( \uplus_{i\in I} P_i^G \Big) \uplus \Big( \uplus_{i\notin I} P_i^G \Big)$ (by Prop. \ref{prop_gmdom}.2). Since $P^G$ is a $k$-partition,  $\Big( \uplus_{i\in I} P_i^G \Big) \uplus \Big( \uplus_{i\notin I} P_i^G \Big) =  \uplus^k\hspace{1pt}V(G)$. Since $A\uplus W = \{u_1,\ldots, u_{\gamma_k(G)} \}$ is a $\gamma_{\{k\}}$-set of $G$, $A\uplus W$
dominates $\uplus^k\hspace{1pt}V(G)$ (by Prop. \ref{prop_gmdom}.1). Therefore
$\lvert B\uplus W \lvert \geq \lvert A\uplus W \lvert$, and $\lvert B\lvert \geq \lvert A\lvert$. The proof of Prop. \ref{prop_domrel}.2 follows similarly. \hfill $\Box$

In the introduction, we stated that the proof of Theorem \ref{thm_kdom} relies on the double-projection technique of Clark and Suen, and also a particular property of binary matrices. Specifically, we have the following proposition:

\begin{proposition}\label{fact1}
Let $M$ be a matrix containing only 0/1 entries. Then one (or both) of the following two statements are true:
\vspace{-6pt}
\begin{enumerate}[label=(\alph*)]
\item each column contains a 1~,
\vspace{-8pt}
\item each row contains a 0~.
\end{enumerate}
\end{proposition}
\begin{proof}
For a proof by contradiction, assume there exists a row (say $i$) which does not contain a 0, and a column (say $j$) which does not contain a 1. 
Then, the entry $M[i,j]$ is neither 0 nor 1, which is a contradiction.
\end{proof}

We are now ready to state and prove the main theorem of the paper.

\begin{theorem} \label{thm_kdom}
For graphs $G$ and $H$, 
$\gamma_{\{k\}}(G) \gamma_{\{k\}}(H) \leq 2k \ \gamma_{\{k\}}(G \Box H)$~.
\end{theorem}

\begin{proof} We begin with the same notation as in Prop. \ref{prop_domrel}. Let $\{u_1 ,\ldots, u_{\gamma_{\{k\}}(G)}\}$ and  
$\{\overline{u}_1 ,\ldots, \overline{u}_{\gamma_{\{k\}}(H)}\}$ be minimum $\{k\}$-dominating multisets of $G, H$, respectively, 
and let $P^G = \{P^G_1, \ldots , P^G_{\gamma_{\{k\}}(G)}\}$ and $P^H = \{P^H_1, \ldots ,$ $ P^H_{\gamma_{\{k\}}(H)}\}$be a 
$k$-partitions of $V(G), V(H)$ respectively, such that $u_i \in P^G_i$ and $P^G_i \subseteq N_G[u_i]$ (and $\overline{u}_j \in P^H_j$ 
and $P^H_{j}\subseteq N_H[\overline{u} _{j}]$, respectively) for $1 \leq i \leq \gamma_{\{k\}}(G)$ and   $1 \leq j \leq \gamma_{\{k\}}(H)$. 
Recall from Example \ref{ex_kpart} that $P^G_i$ and $P^G_j$ may be completely distinct, equal, or overlap in parts. Finally, 
observe that since $P^G$ and $P^H$ are $k$-partitions of $V(G), V(H)$, respectively, $P^G \times P^H$ is a $k^2$-partition of $V(G \Box H)$.

We now describe a notation for uniquely identifying \emph{different} occurrences of the \emph{same} vertex $gh \in V(G \Box H)$ 
in the $k^2$-partition $P^G \times P^H$. Let $I = \{1,\ldots,\gamma_{\{k\}}(G)\}$, $J = \{1,\ldots,\gamma_{\{k\}}(H)\}$, and 
$\overline{V} =\uplus^{k^2}\hspace{1pt}  V(G\Box H) = \uplus_{i\in I} \uplus_{j\in J} (P_i^G \times P_j^H)$. Since $P^G$ is a 
$k$-partition of $V(G)$ and $P^H$ is a $k$-partition of $V(H)$, for each $gh \in V(G\Box H)$, let $f_g : \{1,\ldots,k\} \rightarrow I$ 
and $f_h : \{1,\ldots,k\} \rightarrow J$ be one-to-one functions that identify the $k$ blocks where vertex $g$ appears in the $k$-partition 
$P^G$ (and similarly for $P^H$). Thus, the $k$ copies of $g$ in $P^G$ appear in blocks $P_{f_g(1)}^G, \ldots, P_{f_g(k)}^G$, and 
similarly for $P^H$. Let $(gh)_{sr}$ (for $1 \leq s,r \leq k$) indicate the $sr$-th copy of vertex $gh$ in the $k^2$-partition 
$\overline{V}$ that occurs due to block $P_{f_g(s)}^G \times P_{f_h(r)}^H$.

Let $D_k$ be a minimum $\{k\}$-dominating multiset of $G \Box H$ and $gh \in V(G\Box H)$. Since $D_k$ is a $\gamma_{k}$-set of 
$G \Box H$, there are \emph{at least} $k$ dominators  $d_0, \ldots, d_{k-1}$ of $gh$ in $D_k$ (when counted with multiplicity, thus $d_0,\ldots, d_{k-1}$ are not necessarily distinct dominators). 
We now create a function to assign a specific dominator in $D_k$ (not necessarily unique) to each of the $k^2$ copies of 
$gh \in \overline{V}$. Specifically, let $d: (gh)_{sr} \rightarrow \{d_0,\ldots,d_{k-1}\}$  (for $1 \leq s,r \leq k$) be defined as $d\big((gh)_{sr}\big) = d_{(s+r)\hspace{-3pt}\mod k}$. Note that for $s$ fixed and $1 \leq r \leq k$, the $k$ 
copies of vertex $(gh)_{sr}$ are assigned dominators $\{d_{(s + 1)\hspace{-3pt}\mod k}, \ldots, d_{(s + k)\hspace{-3pt}\mod k}\} = \{d_0,\ldots,d_{k-1}\}$ in $D_k$ (and similarly for $r$ fixed).

We now define a binary matrix corresponding to each of $P_{i}^G \times P_{j}^H$ block (for $i\in I$, $j\in J$) based on the ``type" of 
dominator assigned to a particular vertex $gh$. For $g \in P^G_i$ and  $h \in P^H_j$,  let $s = f_g^{-1}(i)$ and $r = f_h^{-1}(j)$. 
Then, we define the binary $|P^G_i| \times |P^H_j|$ matrix $F_{ij}$ such that:
\[
 F_{ij}(g, h) =
 \begin{cases}
 0 & \text{if } d\big((gh)_{sr}\big) \in N_{\underline{\mathbf{G}} \Box H}(gh)~, \\
 1 & \text{otherwise}~.
 \end{cases}
\]
Observe that for each $i\in I$ and $g \in P^G_i$, even though the function $f_g$ is not onto, the inverse $f_g^{-1}(i)$ is always 
defined since one of the $k$ copies of $g$ in the $k$-partition $P^G$ appears in block $P^G_i$ (and similarly for $f_h^{-1}(j)$). 
Furthermore, observe that $F_{ij}(g, h) = 1$ in two cases: 1) if $d \big((gh)_{sr}\big)$ is vertex $gh$ itself,  
or if $d \big((gh)_{sr}\big)$ dominates $(gh)_{sr}$ via an $\textbf{H}$-edge.

By Prop. \ref{fact1}, each of the binary matrices $F_{ij}$ satisfies one or both of the statements in Prop. \ref{fact1}. 
We will now define a series of multisets based on which of the properties $F_{ij}$ satisfies.
\begin{align*}
Z_i &=\big\{ \uplus (P^G_i\times P^H_j) : F_{ij} \text{ satisfies Prop. \ref{fact1}.a~, $1 \leq j \leq  \gamma_{\{k\}}(H)$} \big\}~, \hspace{10pt} \text{for $1 \leq i \leq \gamma_{\{k\}}(G)$}~,\\
\overline{Z}_j  &=\big\{ \uplus (P^G_i\times P^H_j) : F_{ij} \text{ satisfies Prop. \ref{fact1}.b~, $1 \leq i \leq  \gamma_{\{k\}}(G)$} \big\}~,  \hspace{12pt} \text{for $1 \leq j \leq \gamma_{\{k\}}(H)$}~,\\
N_i &=\big\{\overline{u}_j : F_{ij} \text{ satisfies Prop. \ref{fact1}.a~,  $1 \leq j \leq  \gamma_{\{k\}}(H)$} \big\}~, \hspace{10pt} \text{for $1 \leq i \leq \gamma_{\{k\}}(G)$}~,\\
\overline{N}_j &=\big\{u_i: F_{ij} \text{ satisfies Prop. \ref{fact1}.b~,  $1 \leq i \leq  \gamma_{\{k\}}(G)$} \big\}~, \hspace{12pt} \text{for $1 \leq j \leq \gamma_{\{k\}}(H)$}~,\\
Y_i &=\varPhi_{H}(Z_i) = \big\{ \uplus P^H_j : F_{ij} \text{ satisfies Prop. \ref{fact1}.a~, $1 \leq j \leq  \gamma_{\{k\}}(H)$} \big\}~, \hspace{10pt}\text{for $1 \leq i \leq \gamma_{\{k\}}(G)$}~,\\
\overline{Y}_j &= \varPhi_{G}(\overline{Z}_j) =\big\{ \uplus P^G_i : F_{ij} \text{ satisfies Prop. \ref{fact1}.b~, $1 \leq i \leq  \gamma_{\{k\}}(G)$} \big\}~, \hspace{9pt} \text{for $1 \leq j \leq \gamma_{\{k\}}(H)$}~,\\
S_i &= D_k \cap \Big( \uplus^k\hspace{1pt}\big(P^G_i \times V(H)\big) \Big)~, \hspace{10pt}\text{for $1 \leq i \leq \gamma_{\{k\}}(G)$}~,\\
\overline{S}_j &= D_k \cap \Big( \uplus^k\hspace{1pt}\big(V(G) \times P^H_j\big) \Big)~, \hspace{11pt} \text{for $1 \leq j \leq \gamma_{\{k\}}(H)$}~.
\end{align*}
To clarify the definition of $S_i$, observe that the intersection of two multisets $\{1,1,1,2,2,3\} \cap \{1,1,2,4\} = \{1,1,2\}$. 
We will now prove the following claim.
\begin{claim} \label{cl_dom}
For $i=1,\ldots, \gamma_{\{k\}}(G)$, $\varPsi_{H}(S_i)$ dominates $Y_i$, and
for $j=1,\ldots,\gamma_{\{k\}}(H)$, $\varPsi_{G}(\overline{S}_j)$ dominates
$\overline{Y}_j$.
\end{claim}
\begin{proof} In order to show that $\varPsi_{H}(S_i)$ dominates $Y_i$, we must show that 1) every vertex $y \in Y_i$
is dominated by some vertex $h \in S_i$, and 2) the number of occurrences of $y$-dominators in the multiset 
$\varPsi_H(S_i)$ is greater than or equal to the number of occurrences of $y$ in multiset $Y_i$.

In order to prove (1), consider $y \in Y_i$. By definition, there exists a $j$ such that $y \in P^H_j$ and $F_{ij}$ satisfies Prop. \ref{fact1}.a.
Since column $P^G_i \times y$ of $F_{ij}$ contains a ``1", there exists a $g \in P^G_i$ such that vertex $gy$ is dominated by 
an \textbf{H}-edge (or itself). Let $gh$ be a dominator of $gy$. Thus, there exists an $h \in \varPsi_{H}(S_i)$ such that $h$ dominates $y$.

In order to prove (2), consider $y \in Y_i$. Let $|Y_i|_y = t$. Recall that $Y_i \subseteq \uplus_{j\in J}\hspace{1pt}P_j^H = \uplus^k\hspace{1pt}V(H)$. 
Since $P^H$ is a $k$-partition of $V(H)$, $y$ appears in exactly $k$ blocks of $P^H$. Thus, $t \leq k$.

Let $\{j_1,\ldots,j_t\}$ be such that matrices $F_{ij_1},\ldots,F_{ij_t}$ satisfy Prop. \ref{fact1}.a, and $y$ is contained in 
$P^H_{j_1},\ldots, P^H_{j_t}$. Furthermore, let $g_1,\ldots,g_t \in P^G_i $ be such that $F_{ij_w}(g_w, y) = 1$ (for 
$1 \leq w \leq t$). Then, each of $g_wy$ is dominated by an \textbf{H}-edge (or itself), and given $g_w \neq g_{w'}$ with 
$1 \leq w, w' \leq t$, vertices $g_wy$ and $g_{w'}y$ are dominated by distinct vertices in $D_k$ (and by extension, $S_i$). However, we must now show that two identical vertices $g_{w}y, g_{w'}y$ (i.e., $w = w'$) have dominators with distinct indices in $D_k$. 

Recall that different occurrences of vertex $gh$ in the multiset $\overline{V}$ are denoted as $(gh)_{sr}$, indicating that the 
$sr$-th copy of $gh$ is due to the $P^G_{f_g(s)} \times P^H_{f_h(r)}$ block. In this case, two identical vertices $g_{w}y, g_{w'}y$ 
occur due to the same $P^G_i$ block, but different $P^H_j$ blocks. Furthermore, recall that the dominators $d_0,\ldots,d_{k-1}$ are 
assigned to the different occurrences of $gh$ via the function $d\big((gh)_{sr}\big) = d_{(s + r)\hspace{-3pt}\mod k}$, and that for $s$ fixed and $1 \leq r \leq k$, the $k$ copies of vertex $(gh)_{sr}$ are assigned dominators 
$\{d_{(s + 1)\hspace{-3pt}\mod k}, \ldots, d_{(s + k)\hspace{-3pt}\mod k}\} =\{ d_0,\ldots,d_{k-1}\} \subseteq D_{k}$. Therefore, the two identical vertices $g_wy$ and $g_{w'}y$ will map to dominators with distinct indices in $D_k$.


Finally, recall that the number of occurrences of a given $h \in \varPsi_{H}(S_i)$ is determined due to the \emph{sum} (as opposed 
to the \emph{maximum}). Given a vertex $y \in Y_i$ where $|Y_i|_y = t$, we have demonstrated that there are \emph{at least} $t$ vertices in 
$S_i =  D_k \cap \uplus^k\hspace{1pt}\big(P^G_i \times V(H)\big)$ (when counted with multiplicity) whose projection on $H$ dominates $y$, 
and therefore, there are at least $t$ $y$-dominators appearing in $\varPsi_H(S_i)$ (when counted with multiplicity). 

To conclude, we have demonstrated that 1) every vertex $y \in Y_i$ is dominated by some vertex $h \in S_i$, and 2)
the number of occurrences of $y$-dominators in the multiset $\varPsi_H(S_i)$ is greater than or equal to the number of occurrences of $y$ in multiset $Y_i$.
Therefore, $Y_i$ is dominated by $\varPsi_H(S_i)$.

Similarly, we can demonstrate that $\overline{Y}_j$ is dominated by $\varPsi_G(\overline{S}_j)$. 
\end{proof}

We will now carefully bound the sizes of the sets $N_i, \overline{N}_j, S_i, \overline{S_j}$, etc., in relation to each other. We observe 
that the total number of $P^G_i \times P^H_j$ blocks in the $k^2$-partition of $V(G\Box H)$ is $\gamma_{\{k\}}(G)\gamma_{\{k\}}(H)$. 
Since the binary matrix $F_{ij}$ associated with each these blocks satisfies at least one of the two conditions of Prop. \ref{fact1}, 
we see that
\begin{align*}
\gamma_{\{k\}}(G) \gamma_{\{k\}}(H) \leq \sum\limits_{i=1}^{\gamma_{\{k\}}(G)} \lvert N_i \lvert + \sum\limits_{j=1}^{\gamma_{\{k\}}(H)} \lvert \overline{N}_j \lvert~.
\end{align*}

Since $P_G$ is a $k$-partition of $V(G)$, every $g$ appears in exactly $k$ blocks of $P^G$. Thus, every $gh \in V(G\Box H)$ appears in exactly 
$k$ ``strips"  of $P^G \times V(H)$. Thus,  if $gh \in \big(D_k \cap (P^G_i \times V(H)\big)$, then $gh$ appears in ``strip" 
$D_k \cap \big( \uplus^k\hspace{1pt}P^G_i \times V(H) \big)$ exactly $|D_k|_{gh}$ times. Therefore, when we iterate over all the ``strips", 
we see
\begin{align*}
\sum\limits_{i=1}^{\gamma_{\{k\}}(G)} \lvert S_i \lvert_{gh} &= k |D_k|_{gh}~, \quad \quad
\sum\limits_{i=1}^{\gamma_{\{k\}}(G)} \lvert S_i \lvert =
\sum\limits_{j=1}^{\gamma_{\{k\}}(H)} \lvert \overline{S}_j \lvert = k|D_k| =k \gamma_{\{k\}}(G \Box H)~.
\end{align*}

Since $N_i$ dominates $Y_i$ (Prop. \ref{prop_domrel}.2), and since $\varPsi_H(S_i)$ is another multiset dominating $Y_i$ (Claim \ref{cl_dom}), 
we see by Prop. \ref{prop_domrel}.2 that
\begin{align*}
\lvert \varPsi_H(S_i) \lvert &\geq \lvert N_i\lvert~, \text{~for $1 \leq i \leq \gamma_{\{k\}}(G)$}~, \quad \text{and} \quad 
\lvert \varPsi_G(\overline{S}_j) \lvert \geq \lvert \overline{N}_j \lvert~, \text{~for $1 \leq j \leq \gamma_{\{k\}}(H)$}~.
\end{align*}
Furthermore, since the number of occurrences of a given vertex in the $\varPsi$-projection is determined by the \emph{sum} (as opposed to 
the \emph{maximum}),  $\lvert S_i \lvert = \lvert \varPsi_H(S_i) \lvert$, and 
$\lvert \overline{S}_j \lvert = \lvert \varPsi_G(\overline{S}_j) \lvert$. Therefore,
\begin{align*}
\lvert S_i \lvert &\geq \lvert N_i\lvert~, \text{~for $1 \leq i \leq \gamma_{\{k\}}(G)$}~, \quad \text{ and } \quad
\lvert \overline{S}_j \lvert \geq \lvert \overline{N}_j \lvert~, \hspace{9pt} \text{for $1 \leq j \leq \gamma_{\{k\}}(H)$}~.
\end{align*}
Combining all of these inequalities together, we finally see
\begin{align*}
\gamma_{\{k\}}(G) \gamma_{\{k\}}(H)
 \leq \sum\limits_{i=1}^{\gamma_{\{k\}}(G)} \lvert N_i \lvert + \sum\limits_{j=1}^{\gamma_{\{k\}}(H)} \lvert \overline{N}_j \lvert
\leq \sum\limits_{i=1}^{\gamma_{\{k\}}(G)} \lvert S_i \lvert + \sum\limits_{j=1}^{\gamma_{\{k\}}(H)} \lvert \overline{S}_j \lvert
= 2k \ \gamma_{\{k\}}(G \Box H)~.
\end{align*}
This concludes our proof.
\end{proof}

\section*{Acknowledgements} The authors would like to acknowledge the support of NSF-CMMI-0926618, the Rice University VIGRE program (NSF DMS-0739420 and EMSW21-VIGRE), and the Global Initiatives Fund (Brown School of Engineering at Rice University), under the aegis of SURGE (Summer Undergraduate Research Grant for Excellence), a joint program with the IIT Kanpur and the Rice Center for Engineering Leadership. Additionally, the authors acknowledge the support of NSF DSS-0729251, DSS-0240058, and the Defense Advanced Research Projects Agency
under Award No. N66001-10-1-4040. 



\bibliographystyle{plain}
\bibliography{vizing}

\end{document}